\documentclass[12pt,oneside]{amsart}

\usepackage[a4paper]{geometry}  

\usepackage{amssymb,mathrsfs}

\usepackage[english]{babel}

\usepackage{mathptmx}

\usepackage{hyperref}
\hypersetup{unicode,pdfauthor={J\'an \v Spakula}}

\theoremstyle{plain}    
 \newtheorem{theorem}{Theorem}[section]
\theoremstyle{definition}
 \newtheorem{definition}[theorem]{Definition}

\theoremstyle{plain}    
 \newtheorem{proposition}[theorem]{Proposition}
 \newtheorem{corollary}[theorem]{Corollary}
 \newtheorem{lemma}[theorem]{Lemma} 
 
 \theoremstyle{definition}
\theoremstyle{remark}
 \newtheorem{remark}[theorem]{Remark}
 \newtheorem{example}[theorem]{Example}

\def\N{\mathbb{N}}

\def\C{\mathbb{C}}

\def\K{\mathcal{K}}
\def\B{\mathcal{B}}

\def\X{\mathcal{X}}

\DeclareMathOperator{\diam}{diam}

\DeclareMathOperator{\id}{id}

\DeclareMathOperator{\spa}{span}

\let\eps\varepsilon

\def\D{\mathscr{D}}

\def\P{\mathscr{P}}
\hypersetup{pdftitle={Uniform version of Weyl--von Neumann theorem}}

\title{Uniform version of Weyl--von Neumann theorem} 
\author{J\'an \v Spakula}

\address{Mathematisches Institut, Universit\"at M\"unster, Einsteinstr.\ 62,
    48149 M\"unster, Germany}
\email{jan.spakula@uni-muenster.de}

\begin{document}

\begin{abstract}
  We prove a ``quantified'' version of the Weyl--von Neumann theorem, more
  precisely, we estimate the ranks of approximants to compact operators
  appearing in the Voiculescu's theorem applied to commutative algebras. This
  allows considerable simplifications in uniform $K$-homology theory, namely
  it shows that one can represent all the uniform $K$-homology classes on a
  fixed Hilbert space with a fixed *-representation of $C_0(X)$, for a large
  class of spaces $X$.
\end{abstract}

\maketitle

\section{Introduction}

Voiculescu's theorem \cite{voiculescu:voiculescu-theorem} states that
whenever one has a non-degenerate representation $\pi:E\to \B(H_\pi)$ of a
separable unital C*-algebra $E$ and a completely positive map $\rho:E\to
\B(H_\rho)$ with the property that $\pi(e)\in \K(H)\implies \rho(e)=0$ for
every $e\in E$, then there exists a unitary $V: H_\rho\to H_\pi$, such that
$\rho(e)-V^*\pi(e)V\in \K(H_\rho)$. Loosely worded, when $\pi:E\to \B(H_\pi)$
is sufficiently big, then any other representation of $E$ is a compression of
$\pi$ modulo compacts. In this paper, we prove a ``quantified'' version of
this statement for commutative algebras $E$ and for $*$-homomorphisms
$\rho$. More precisely, if we consider an class $\X$ of compact spaces $X$
having jointly locally bounded geometry (defined below), we can obtain bounds
on the ranks of approximants of $\rho(f)-V^*\pi(f)V$ which depend only on the
type of the function $f\in C(X)$ and is independent of the space $X\in\X$
itself.

The special case of Voiculescu's theorem, when the algebra $E$ in question is
commutative, is usually called the Weyl--von Neumann theorem and stated in
the form that any normal operator on a separable Hilbert space can be
expressed as a diagonal operator plus a compact operator, and additionally
the norm of the compact can be made arbitrarily small.

The motivation and an application of our theorem come from considerations in
analytic $K$-homology and coarse geometry. The analytic $K$-homology theory
is a generalized homology theory on the category of locally compact Hausdorff
topological spaces (with proper maps); in its full generality even a
contravariant functor on the category of separable C*-algebras
\cite{kas:orig,higson-roe:analKhomol}. The connection to coarse geometry and
index theory comes from the existence of a coarse index map from analytic
$K$-homology of a locally compact space into the $K$-theory of its Roe
C*-algebra \cite{roe:CBMS}. The assertion that this map is an isomorphism has
notable applications for instance when the space $X$ is question is a Rips
complex of a Cayley graph of a finitely generated group $\Gamma$, namely it
implies the Novikov conjecture for $\Gamma$
\cite{yu:CBC-for-embeddable}. Also, since analytic $K$-homology of spaces is
``computable'' (by means of exact sequences), this provides a way to compute
$K$-theory of Roe C*-algebras.

The Voiculescu's theorem provides a way to simplify the definition and proofs
in analytic $K$-homology as follows: The cycles for a C*-algebra $A$ are
represented by Fredholm modules $(H,\phi,F)$, where $H$ is a (graded) Hilbert
space, $\phi:A\to\B(H)$ is a *-homomorphism (of degree 0), $F\in \B(H)$ (has
degree 1) and $[F,\phi(a)]\in \K(H)$, $(F^2-1)\phi(a)\in \K(H)$ for all $a\in
A$. By introducing a suitable equivalence relation on Fredholm modules one
obtains the $K$-homology group. A consequence of the Voiculescu's theorem is
that we can fix any Hilbert space $H_0$ with a representation
$\phi_0:A\to\B(H_0)$ which misses $\K(H_0)\setminus\{0\}$ and then any
$K$-homology class of $A$ can be represented as a Fredholm module of the form
$(H_0,\phi_0,\cdot)$.

We present here an application of the main result of this paper to uniform
$K$-homology: Uniform $K$-homology is a version of analytic $K$-homology for
locally compact separable metric spaces $X$ with bounded geometry, defined in
\cite{spakula:unifKhomol}, from which there is an index map into the
$K$-theory of the uniform Roe C*-algebra $C^*_u(X)$. Furthermore, one can
characterize amenability of $X$ in terms of this theory. A consequence of the
main result of this paper is that one can work on a fixed uniform Fredholm
module in the uniform $K$-homology (as explained in the previous paragraph
for the ordinary $K$-homology) for locally compact spaces $X$ which have
bounded geometry not only on the large scale (i.e.\ in the coarse geometric
understanding of bounded geometry), but also on the small scale.

\smallskip
\textbf{Acknowledgment:} The author would like to thank Jozef
Hale\v s for his valuable suggestions concerning topological questions.

\section{Result}

The following definition formalizes the notion of having bounded geometry on
the small scale. Naturally, any single compact metric space $X$ by itself has
bounded geometry in any sense, but the point of our main Theorem
\ref{thm:uniform-WvNB} is to obtain estimates which are independent of the
space $X$. The notion defined below allows us to specify classes of spaces
for which this is possible.

\begin{definition}
  We say that compact metric spaces in a class $\X$ have \emph{jointly
    locally bounded geometry}, or just shortly that $\X$ is
  \emph{admissible}, if the following property holds: Given $\eps>0$, there
  exists $N\geq0$, such that for any $X\in\X$ there exists an $\eps$-net in
  $X$ of cardinality at most $N$.
\end{definition}

\begin{example}\label{ex:nice-spaces}
  If $Y$ is a complete Riemannian manifold with bounded geometry (in the
  sense of Roe \cite{roe:index-on-open}), or a uniformly locally finite
  simplicial complex of finite dimension (for instance a Rips complex of a
  finitely generated group), then for any $R>0$, the collection closed of
  $R$-balls in $Y$ comprises an admissible class of spaces.
\end{example}

\begin{remark}
  Note that bounding only the diameter and the (covering) dimension of a
  space does not yield an admissible class. Just consider $1$-dimensional
  simplicial complexes with one ``central'' vertex and $n$ different edges
  attached to it. The size of $\eps$-nets in such a space grows with $n$, but
  neither the diameter nor the dimension does not.
\end{remark}

For a metric space $X$ and $L\geq0$, we denote 
$$
C_L(X)=\left\{ f:X\to\C\mid f\text{ is $L$-Lipschitz}, \|f\|_\infty\leq 1
\right\}
$$
and by $C(X)$ the C*-algebra of continuous complex-valued functions on
$X$.

\begin{lemma}\label{lem:nets-in-CL}
  If $\X$ is admissible, then the following also holds: Given $L\geq0$ and
  $\eps>0$, there exists $N\geq0$, such that for any $X\in\X$ there exists an
  $\eps$-net of simple Borel functions for $C_L(X)$ containing at most $N$
  elements.
\end{lemma}

\begin{proof}
  The fact that $C_L(X)$ is compact follows for instance from Arzela--Ascoli
  theorem, but the point is to have a bound on the number of elements in an
  $\eps$-net.

  Given $L\geq 0$ and $\eps>0$, we choose $\eps_1>0$ small enough and $K$
  large enough, so that $\frac1K+L\eps_1<\eps$. Now for $\eps_1>0$, the
  admissibility of $\X$ provides us with an upper bound on the size of
  $\eps_1$-nets: denote it by $M\geq0$. Taking any $X\in\X$ and an
  $\eps_1$-net $E$ in $X$ with $|E|\leq M$, we consider any Borel partition
  $\{D_x\mid x\in E\}$ of $X$, such that $x\in D_x$ and
  $\diam(D_x)\leq\eps_1$ for each $x\in E$. For instance, we can obtain such
  a partition by considering the ``closest point in $E$'' map.

  For an integer $K\geq1$, consider the collection of simple functions of the
  form $s=\sum_{x\in E}\frac{i_x}K\chi_{D_x}$, where each $i_x\in \{0,\dots,
  K\}$. There is at most $N=M^{K+1}$ of them. Furthermore, taking any $f\in
  C_L(X)$, there is at least one of them that approximates $f$ within
  $\frac1K$ at each $x\in E$, and then by the Lipschitz property, it
  approximates $f$ within $\frac1{K}+L\eps_1<\eps$ at any $y\in X$.
\end{proof}

\begin{theorem}\label{thm:uniform-WvNB}
  For an admissible class $\X$ of compact metric spaces there exists a
  function $M_\X:[0,\infty)\times(0,1]\to\N$, which satisfies the following:
  For any $X\in\X$, any *-representations $\pi:C(X)\to \B(H_\pi)$,
  $\rho:C(X)\to \B(H_\rho)$, such that $\pi$ is injective and
  $\pi(C(X))\cap\K(H_\pi)=\{0\}$, there exists an isometry $V:H_\rho\to
  H_\pi$, such that for any $L\geq0$ and $\eps>0$, all the operators $V^*
  \pi(f)V -\rho(f)$, $f\in C_L(X)$, are within $\eps$ from an operator with
  rank $M_\X(L,\eps)$.
\end{theorem}

\begin{proof}
  The strategy is to follow parts of the proof that any representation
  $\rho:C(X)\to\B(H_\rho)$ is diagonalizable modulo compacts (this is the
  Weyl--von Neumann theorem), with some estimates on the ranks of the
  finite-rank approximants to the compacts involved. The diagonalizations of
  two representations of $C(X)$ will then provide us with an isometry. We
  then need to analyze the diagonalizations to prove that this construction
  will give us the estimates we require.

  For a representation $\rho:C(X)\to \B(H_\rho)$, we shall without mention use
  the fact that it extends to a representation of the algebra of bounded
  Borel functions on $X$ into $\B(H_\rho)$. For the sake of brevity, we shall
  write $F^\rho\in\B(H_\rho)$ for $\rho(F)$, where $F$ is a bounded Borel
  function on $X$.

  We fix sequences $(L_k)_{k\in\N}$, $L_k\nearrow\infty$, and
  $(\eps_k)_{k\in\N}$, $\eps_k\searrow0$, of positive reals. From the
  admissibility of $\X$ it follows that there exists a sequence
  $(S_k)_{k\in\N}$ of positive integers, such that the following construction
  can be executed: An inductive application of the previous lemma provides us
  with a sequence $(I_k)_{k\in\N}$, $I_k\leq S_k$ of non-negative integers,
  and for any $X\in\X$ with a sequence of (automatically commuting)
  projections $\P_X=\{P_1=\chi_{X_1},\dots\}$, such that $C(X)\subset
  A=\overline{\spa(\P_X)}$ and that $C_{L_k}(X)$ is contained in the
  $\eps_k$--neighborhood of $\spa\{P_1,\dots,P_{I_k}\}$ (in the $\sup$-norm
  on bounded Borel functions; whence it is also true after mapping through
  any representation $\rho:C(X)\to\B(H_\rho)$). For the sake of simplifying
  the notation later on in this proof, we shall arrange that the projections
  in $\P_X$ enjoy some extra properties. Namely, we assume that the algebras
  the $A_k$ generated by $\{P_1,\dots,P_{I_k}\}$ have in fact a linear basis
  $\{P_{I_{k-1}+1},\dots,P_{I_k}\}$ and that
  $P_{I_{k-1}+1}+\dots+P_{I_k}=1$. This can be arranged by choosing each
  successive partition $\{X_{I_{k-1}+1}, \dots, X_{I_k}\}$ of $X$ in the
  proof of the previous lemma so that each $X_j$, $I_{k-1}<j\leq I_k$ is a
  subset of some $X_{j'}$, $j'\leq I_{k-1}$ chosen previously. Furthermore,
  we assume that the sets $X_j$ have nonempty interior. These conditions can
  be achieved by carefully choosing the partitions, for instance by
  inductively using admissibility for the closures of members of the previous
  partition.

  \smallskip

  Since the property of having jointly locally bounded geometry is preserved
  under taking closed subspaces of metric spaces, we can assume that $\X$ is
  closed under this operation, and hence that $\rho$ is without loss of
  generality injective. Furthermore, we assume that $\rho$ is unital.

  The proof of diagonalizability of $\rho$ modulo compacts is similar to the
  proofs of the Weyl--von Neumann theorem, e.g.\ \cite[Theorem
  II.4.1]{davidson:c*alg-by-example} and \cite[Theorem
  2.2.5]{higson-roe:analKhomol}. Fix an orthonormal basis $\{e_1,e_2,\dots\}$
  for $H_\rho$.
  Denote $R_k=\{P_{I_{k-1}+1},\dots,P_{I_k}\}$. Note that the projections in
  $R_k$ are mutually orthogonal, their sum is $1$ and the projections in
  $R_j$, $j\leq k$, are sums of projections from $R_k$. We define subspaces
  $E_k^\rho$ of $H_\rho$ by setting
  $$
  E_k^\rho=\spa\{P^\rho e_j\mid 1\leq j\leq k,\, P\in R_k\}.
  $$
  Then $E_k^\rho$ is an increasing sequence of finite-dimensional spaces
  ($\dim(E_k^\rho)\leq k(I_k-I_{k-1})$). Clearly $\spa\{e_1,\dots,e_k\}\subset
  E_k^\rho$, hence $\cup_kE_k^\rho$ is dense in $H_\rho$.

  Consider now the algebras $A_k$ generated by $R_k$ (which is also a linear
  basis of $A_k$). Since both $E_k^\rho$ and $E_{k+1}^\rho$ are invariant for
  $\rho(A_k)$, the subspace $E_{k+1}^\rho\ominus E_k^\rho$ decomposes as the
  orthogonal sum of subspaces $P^\rho(E_{k+1}^\rho\ominus E_k^\rho)$, $P\in
  R_k$. Thus we may choose an orthonormal basis of $E_{k+1}^\rho\ominus
  E_k^\rho$ which respects this decomposition. This diagonalizes $\rho(A_k)$
  on $E_{k+1}^\rho\ominus E_k^\rho$. Note that since $A_j\subset A_k$ for
  $j\leq k$, it also diagonalizes the $\rho$-images of all these algebras. In
  this manner, we obtain an orthonormal basis of the whole $H_\rho$. In
  particular, each $T^\rho$, $T\in A_k$, is eventually diagonal in this basis
  and it differs from a diagonal operator by at most a
  rank-$(\dim(E_k^\rho))$ operator. Furthermore, if we express $T=\sum_{P\in
    R_k}c_PP$, then the entries on the diagonal of $T^\rho$ in this basis are
  eventually just appropriate coefficients $c_P$.

  Let us remark here (although we shall not need in the rest of the proof),
  that the the operators $D_n=P_n^\rho(1-F_n)$, $n\geq 1$, generate a diagonal
  C*-algebra $\D$, such that $\rho(C(X))\subset \D+\K(H_\rho)$.
  
  \smallskip

  Note that in the case $\rho=\pi$, i.e.\ when $\rho(C(X))\cap
  \K(H_\pi)=\{0\}$, we can choose the orthonormal basis $\{e_1,\dots\}$ in
  such a way that each $\dim(E_k^\rho)$ is maximal possible (for a given $X$ and
  $\P_X$). We choose it inductively. First note that without loss of
  generality we can assume that all the projections $P_j^\pi\in\pi(\P_X)$ are
  infinite (since $P_j=\chi_{X_j}$ with $X_j$ having nonempty interior), and
  also $1-P_j^\pi$ is infinite.

  Let us proceed to the induction which results in
  choosing $e_1$: We choose $e_1^{(1)}$ so that all the vectors from the set
  $V_1=\{P_j^\pi e_1^{(1)}\mid 1\leq j\leq I_1\}$ are nonzero. In the $k$-th step,
  we assume that we have already picked vectors
  $e_1^{(1)},\dots,e_1^{(k-1)}$, such that if we denote
  $V_{k-1}=\{P^\pi(e_1^{(i)}+\dots+e_1^{(k-1)})\mid P\in R_{k-1}\}$ and
  $T_{k-1}=\min(\|v\|,v\in V_{k-1})$, then the vectors in $V_{k-1}$ are
  non-zero, i.e.\ $T_{k-1}>0$. Since the projections involved are infinite,
  we can choose $e_1^{(k)}$ in such a way that all the vectors
  $P^\pi(e_1^{(i)}+\dots+e_1^{(k-1)}+e_1^{(k)})$, $P\in R_k$, are again non-zero
  and $\|e_1^{(k)}\|\leq \frac13T_{k-1}$. Finally let $e_1$ the multiple of
  $\sum_{k\geq1}e_1^{(k)}$ with norm $1$. By the bounds on the norms of
  $e_1^{(k)}$'s, this sum converges, and moreover $P_j^\pi e_1\not=0$ for any $j$
  (since $\|\sum_{k\geq m+1}e_1^{(k)}\|\leq\frac23T_{m}$, thus $P^\pi e_1\not=0$
  follows from $\|P^\pi(e_1^{(1)}+\dots+e_1^{(m)})\|\geq T_{m}$ for $P\in
  R_m$). Furthermore for each $m$, the vectors $P^\pi e_1$, $P\in R_m$, are
  automatically linearly independent since the projections in $R_m$ are
  orthogonal.
  
  We repeat this process for choosing $e_2,\dots$. The only change is that
  when choosing $e_i^{(j)}$, we need to also ensure that it is itself
  orthogonal to $e_1,\dots,e_{i-1}$ and that
  $P^\pi(e_i^{(1)}+\dots+e_i^{(j)})$, $P\in R_j$, are orthogonal to the
  so-far chosen $P^\pi e_1,\dots,P^\pi e_{i-1}$. This is possible since we
  are always excluding only finite-dimensional subspaces and the projections
  involved are infinite. The result of this process is an orthonormal basis
  of $H_\pi$, such that $\dim(E_k^\rho)=k(I_k-I_{k-1})$, which is maximal
  possible. Moreover, also the individual $\dim(P^\pi(H_\pi)\cap
  (E_{k+1}^\pi\ominus E_k^\pi))$,
  $P\in R_k$ are maximal possible among all $\dim(P^\rho(H_\rho)\cap
  (E_{k+1}^\rho\ominus E_k^\pi))$
  for all the representations $\rho:C(X)\to \B(H_\rho)$.

  \smallskip

  By the conclusion in the previous paragraph, there exists an isometry
  $V:H_\rho\to H_\pi$, constructed so that within each $P^\rho(H_\rho)\cap
  (E_{k+1}^\rho\ominus E_k^\rho)$, we send the chosen basis (which
  diagonalizes $\rho(A_k)$) to the (possibly part of) the chosen basis of
  $P^\pi(H_\pi)\cap (E_{k+1}^\pi\ominus E_k^\pi)$.

  It remains to be shown that this isometry will indeed satisfy the estimates
  we require. It is clear from the construction that $V*T^\pi V-T^\rho$ for
  $T\in A_k$ is a finite-rank operator, with rank at most $\dim(E_k^\rho)\leq
  k(I_k-I_{k-1})\leq kS_k$. By our choices, any $f\in C_{L_k}(X)$ is at most
  $\eps_k$-far from some $T\in A_k$, hence $V*\pi(f)V-\rho(f)$ is at most
  $2\eps_k$-far from a rank-$2kS_k$ operator. So to finish the proof, we just
  define the function $M_\X(L,\eps)=2kS_k$, where $k$ is large enough so that
  $L\leq L_k$ and $\eps\geq \eps_k$.
\end{proof}

\section{Application to uniform $K$-homology}

We apply Theorem \ref{thm:uniform-WvNB} to show that any for ``nice'' locally
compact metric spaces $X$ (namely the ones with both locally and coarsely
bounded geometry), we can fix a suitable Hilbert space $H$ and a
representation $\phi:C_0(X)\to\B(H)$ and then represent any class in the
uniform $K$-homology $K^u_*(X)$ of $X$ as a uniform Fredholm module of the
form $(H,\phi,F)$. The class of ``nice'' spaces includes open manifolds with
bounded geometry and Rips complexes of finitely generated discrete groups
(see Example \ref{ex:nice-spaces}).

\begin{definition}
  We say that a metric space $X$ has \emph{locally bounded geometry}, if it
  admits a countable Borel decomposition $X=\cup_{i\in I}X_i$, such that
  $\X=\{\overline{X_i}\mid i\in I\}$ is an admissible class of
  compact metric spaces and that each $X_i$ has nonempty interior.\\
  We say that $X$ has \emph{coarsely bounded geometry}, if it contains a
  uniformly discrete subset $Y$, such that $\sup_{x\in X}d(x,Y)<\infty$ and
  which has bounded geometry in a sense that for each $R\geq0$, $\sup_{y\in
    Y}\#B(y,R)<\infty$.
\end{definition}

The tool we use to show that Theorem \ref{thm:uniform-WvNB} implies the above
statement in uniform $K$-homology is the notion of a uniformly covering
isometry. We recall its definition \cite[Definition 4.7]{spakula:unifKhomol};
it is a ``quantified version'' of the notion of a covering unitary
\cite[Definition 5.2.2]{higson-roe:analKhomol}.

\begin{definition}\label{def:unifcovers}
  Let $X$ and $Z$ be metric spaces, let $\varphi:C_0(X)\to C_0(Z)$ be a
  *-ho\-mo\-mor\-phi\-sm, $\phi_X:C_0(X)\to \B(H_X)$ and
  $\phi_Z:C_0(Z)\to\B(H_Z)$ be *-representations. We say that an isometry $V:
  H_Z\to H_X$ \emph{uniformly covers} $\varphi$, if for every $\eps>0$,
  $R,L\geq0$ there exists $M\geq0$, such that for every $f\in C_0(X)$ which
  is $L$-Lipschitz and has support of diameter at most $R$, there exists an
  operator with rank at most $M$, which is at most $\eps$-far from
  $V^*\phi_X(f)V-\phi_Z(\varphi(f))$. [This relation is denoted in
  \cite{spakula:unifKhomol} by
  $V^*\phi_X(\cdot)V\sim_{lua}\phi_Z(\varphi(\cdot))$; and it is a sharpening
  of $\sim$ which expresses that the difference of the two operators is
  compact.]
\end{definition}

\begin{corollary}\label{cor:unifcov}
  Let $X$ be a locally compact metric space which has both locally and
  coarsely bounded geometry. Let $\pi:C_0(X)\to \B(H)$ be a *-representation
  that misses $\K(H)\setminus\{0\}$. Then for any non-degenerate
  *-representation $\rho:C_0(X)\to \B(H_\rho)$ that misses the compacts, there
  exists an isometry $V:H_\rho\to H$ that uniformly covers the identity map
  $\id:C_0(X)\to C_0(X)$.
\end{corollary}

\begin{remark}
  For a compact space $X$ any covering isometry is automatically uniformly
  covering; this is because we can disregard $R$ in this case and (as in
  Lemma \ref{lem:nets-in-CL}) each $C_L(X)$ is compact, hence it is
  sufficient to consider only finitely many compacts for approximation by
  finite-rank operators.
\end{remark}

\begin{proof}
  By assumption on $X$, there is a Borel decomposition $X=\cup_{i\in \N}X_i$,
  such that $\X=\{X_i\mid i\in\N\}$ is an admissible class and each $X_i$ has
  nonempty interior. Also note that it follows from the admissibility that
  $R_0=\sup_{i\in\N}(\diam(X_i))<\infty$.

  Notice that we can apply the whole proof of the Theorem
  \ref{thm:uniform-WvNB} also to the family $\X'=\{X_i\mid i\in\N\}$, where
  the algebras of functions $C(X_i)$ on $X_i$'s that we consider are the
  algebras of those functions that are restrictions of functions from
  $C_0(X)$ (the conclusion of Lemma \ref{lem:nets-in-CL} is still true for
  $\X'$).

  Denoting $P_i=\chi_{X_i}$, $H_i=\pi(P_i)H$ and
  $H_{\rho,i}=\rho(P_i)H_\rho$, we have decompositions $H=\oplus_{i\in\N}H_i$
  and similarly $H_\rho=\oplus_{i\in\N}H_{\rho,i}$. We may assume that $\pi$
  is non-degenerate. The representations $\pi$ and $\rho$ restrict to
  families of representations $\pi_i=\pi|:C(X_i)\to \B(H_i)$ and
  $\rho_i=\rho|:C(X_i)\to \B(H_{\rho,i})$. Applying Theorem
  \ref{thm:uniform-WvNB} we obtain a function $M$ and isometries
  $V_i:H_{\rho,i}\to H_i$, such that for any $\eps>0$ and $L\geq0$, taking
  any $f\in C_L(X_i)$ implies that $V_i^*\pi_i(f)V_i-\rho_i(f)$ is at most
  $\eps$-far from a rank-$M(\eps,L)$ operator. If we now denote
  $V=\oplus_{i\in\N}V_i:H_\rho\to H$, we obtain an isometry which satisfies
  the condition for uniform covering of $\id:C_0(X)\to C_0(X)$ for functions
  that are supported within some $X_i$.

  The final step is to realize that coarsely bounded geometry condition
  implies that given $R\geq0$, there is an upper bound on how many $X_i$'s
  can meet any ball of radius $R$ in $X$. Consequently the isometry $V$ is as
  required.
\end{proof}

Let us now come the application to uniform $K$-homology. For precise
definitions of uniform Fredholm modules and uniform $K$-homology groups, see
\cite[Section 2]{spakula:unifKhomol}. Let us only note here that they mimic
the usual definitions of Fredholm modules and analytic $K$-homology groups,
but where the relations ``is compact'' is replaced by the ``quantified
version'' of being compact, in the same spirit as in Definition
\ref{def:unifcovers}. The following proposition explains the connection
between uniformly covering isometries and uniform Fredholm modules.

\begin{proposition}\label{prop:unifcov}
  Let $X$ be a locally compact metric space and let $\phi:C_0(X)\to \B(H)$
  and $\rho:C_0(X)\to \B(H_\rho)$ be *-representations. Assume that there
  exists an isometry $V:H_\rho\to H$ which uniformly covers the identity map
  $\id:C_0(X)\to C_0(X)$. Then any uniform Fredholm module $(H_\rho,\rho,F)$
  is equivalent in the uniform $K$-homology group $K_*^u(X)$ to a
  uniform Fredholm module $(H,\phi,F')$.
\end{proposition}

\begin{proof}
  This follows from the results proved in \cite{spakula:unifKhomol}, namely
  Proposition 4.3 (which explains the description of uniform $K$-homology
  groups in terms of the $K$-theory of ``dual'' C*-algebras), Lemma 5.4
  (which shows that uniformly covering isometry induces a map between the
  dual C*-algebras) and Proposition 4.9 (where the ``partial'' uniform
  $K$-homology groups with specified $*$-representation $\phi$ are put
  together). See also the discussion preceding Lemma 4.8.

  Compare also \cite[Proposition 8.3.14]{higson-roe:analKhomol}, which
  supplies a more direct proof in the case when the uniformly covering
  isometry is actually onto.
\end{proof}

Putting together Corollary \ref{cor:unifcov} and Proposition
\ref{prop:unifcov}, we obtain the final result.

\begin{corollary}
  Let $X$ be a locally compact metric space which has both locally and
  coarsely bounded geometry. Let $\pi:C_0(X)\to \B(H)$ be a *-representation
  that misses $\K(H)\setminus\{0\}$. Then any uniform $K$-homology class over
  $X$ can be represented by a uniform Fredholm module of the form $(H,\pi,T)$
  for some $T\in\B(H)$.
\end{corollary}



\end{document}